\documentclass[12pt,reqno]{amsart}
\usepackage{amsmath, amsfonts, amscd, amssymb,amsthm}
\usepackage[all]{xy}
\usepackage{comment}
\usepackage[a4paper, top = 1.4in, bottom = 1.3in, 
left = 1.5in, right = 1.5in]{geometry}

\newtheorem{thm}{Theorem}[section]

\newtheorem{prop}[thm]{Proposition}
\newtheorem{rmk}[thm]{Remark} 
\newtheorem{cor}[thm]{Corollary}

\newcommand{\R}{{\mathbb R}}
\newcommand{\C}{{\mathbb C}}
\newcommand{\Z}{{\mathbb Z}}
\newcommand{\N}{{\mathbb N}}
\newcommand{\Sp}{\textrm{Sp}}
\newcommand{\lan}{\langle}
\newcommand{\ran}{\rangle}
\newcommand{\mrm}[1]{\mathrm{#1}}
\newcommand{\mcal}[1]{\mathcal{#1}}
\renewcommand{\mod}{{\, \rm mod \, }}
\newcommand{\psmat}[4]{\bigl( \begin{smallmatrix} #1 & #2 
\\ #3 & #4 \end{smallmatrix} \bigr)}

\title[Fourier-Jacobi coefficients]{A Note on Fourier-Jacobi 
coefficients of Siegel modular forms}

\author{Sanoli Gun}
\address{Sanoli Gun$\phantom{mmmmmmmmmm
mmmmmmmmmmmmmmmmmmmmmmmmmmmmmmmmmmm}$\\
Institute of Mathematical Sciences\\
C.I.T Campus, Taramani \\
Chennai 600113, 
India.}
\email{sanoli@imsc.res.in}

\author{Narasimha Kumar}
\address{Narasimha Kumar $\phantom{mmmmmmmmmm
mmmmmmmmmmmmmmmmmmmmmmmmmmmmmmmmmmm}$\\
Indian Institute of Technology Hyderabad \\
Ordnance Factory Estate \\
Yeddumailaram 502205 \\
India.}
\email{narasimha.kumar@iith.ac.in}

\begin{document}

\begin{abstract}
Let $F$ be a Siegel cusp form of weight $k$ and genus $n>1$ with 
Fourier-Jacobi coefficients $f_m$. In this article, we estimate 
the growth of the Petersson norms 
of $f_m$, where $m$ runs over an arithmetic progression.
This result sharpens a recent result of Kohnen in~\cite{WK1}. 
\end{abstract}



\maketitle

\section{Introduction}

\smallskip

Let $\mcal{H}_n$ be the Siegel upper half-plane 
of genus $n \ge 1$ and $\Gamma_n : = \Sp_n(\Z)$ be the full 
Siegel modular group. Also let $S_k(\Gamma_n)$ be the space of 
Siegel cusp forms of weight $k$ on $\Gamma_n$.

For $Z \in \mcal{H}_n$, write $Z= \psmat{\tau}{z^t}{z}{\tau^{\prime}}$,
where $\tau \in \mcal{H}_{n-1}$, $z \in \C^{n-1}$
and $\tau^{\prime} \in \mcal{H}_1$. If $F \in S_k(\Gamma_n)$ 
with $n>1$, the Fourier-Jacobi
expansion of $F$ relative to the maximal parabolic 
group of type $(n-1, 1)$ is of the form
$$ 
F(Z) = \sum_{m \geq 1} f_m(\tau,z) e^{2\pi i m \tau^{\prime}}. 
$$
The functions  $f_m$ belong to the space $J_{k,m}^{\mrm{cusp}}$ 
of Jacobi cusp forms 
of weight $k$, index $m$
and of genus $n-1$, i.e., invariant under the Jacobi group 
$\Gamma_{n-1}^J := \Gamma_{n-1} \ltimes \Z^{n-1} \times \Z^{n-1}$. 
For $f,g \in J_{k,m}^{\mrm{cusp}}$, the inner product of $f$ and 
$g$ is defined by
$$ 
\langle f,g \rangle = \int_{\Gamma_{n-1}^J \backslash \mcal{H}^{n-1} 
\times \C^{n-1}} f(\tau,z) \overline{g(\tau,z)}(\mrm{det}v)^{k-n-1} 
e^{-4 \pi mv^{-1}[y^{t}]} dudvdxdy,
$$
where $\tau=u+iv, z=x+iy$.

Let $a, q \geq 2$ be natural numbers with $(a,q)=1$.  
In~\cite[Thm. 1] {BBK}, B\"ocherer, Bruinier and Kohnen showed that 
for any non-zero function $F$ in $S_k(\Gamma_n) (n>1)$ with 
Fourier-Jacobi coefficients $f_m$, 
there exist infinitely many $m \in \N$  with $m \equiv a \pmod q$ such that 
$\lan f_m,f_m \ran \not = 0$. In this article, 
we prove the existence of infinitely many $m \in \N$ with 
$m \equiv a \pmod q$ such that $\lan f_m,f_m \ran> c_{F,q} m^{k-1}$ 
(see Theorem~\ref{main-thm}). This also improves a recent result of 
Kohnen \cite{WK1} about existence of infinitely many $m \ge 1$ 
such that $\lan f_m, f_m \ran > c_F m^{k-1}$.
In order to prove our result, we combine the techniques of
\cite{BBK} and \cite{WK1}.

\smallskip

\section{Preliminaries}

\smallskip

Let $F$ be a non-zero cusp form in $S_k(\Gamma_n) (n>1)$ with 
Fourier-Jacobi coefficients $\{ f_m\}_{m \in \N}$. 
By the works of Kohnen and Skoruppa \cite{KS} 
and of Krieg \cite{AK}, we know that 
$\lan f_m, f_m \ran \ll_F m^k$ (the constant in $\ll$ depends only on $F$). 
Hence for natural numbers $a, q$ with $(a, q)=1$, the Dirichlet series 
\begin{equation*}
D(s ; a, q, F) : = \underset{\underset{m \equiv a \mod q}
{m \ge 1}}{\sum} \frac{\lan f_m, f_m \ran}{m^s}
\end{equation*}
converges for $s \in \C$ with $\Re(s) > k+1$.

\begin{prop}\label{Main-Prop}
Let $a, q >1$ be natural numbers with $(a,q) =1$ and 
$F \in S_k(\Gamma_n)$, where $n > 1$. Then the
Dirichlet series $D(s;a,q, F)$ converges for
$\Re(s)> k$ and has a simple pole at $s=k$. 
Moreover, it vanishes at $s = 0, -1, -2, \cdots$. 
\end{prop}

\begin{proof}
Let $\chi$ be a Dirichlet character modulo $q$. Then the
Dirichlet series 
$$
D(s, \chi, F) := \sum_{m \ge 1} \frac{\chi(m)\lan f_m, f_m \ran}{m^s}
$$
converges for $s \in \C$ with $\Re(s)\gg 0$.
Let $\chi_0$ be the principal Dirichlet character modulo $q$.
For $\chi \not = \chi_0$, we know that 
the completed Dirichlet series
\begin{equation*}
D^*(s, \chi , F) := \left(\frac{2\pi}{q}\right)^{-2s} \Gamma(s) 
\Gamma(s-k+n) L(2s -2k + 2n, \chi^2)~D(s, \chi, F)
\end{equation*}
extends to a holomorphic function on $\C$ 
(see \cite{KKS} and \cite{WK1}
for details). But when $\chi = \chi_0$, 
the completed Dirichlet
series $D^*(s, \chi_0, F)$ has a meromorphic continuation 
to $\C$ with a simple real pole at $s=k$ 
(see~\cite[page 495]{KKS} and 
the remark in page $7$ of~\cite{BBK}). 

We know  if $\chi^2 \ne \chi_0$, then the real 
zeros of $L(s, \chi^2)$ are 
at $s = 0, -2, -4, \cdots$ since $\chi^2$ is an even character.
Also the poles of $\Gamma(s)$ are at  $s = 0 ,-1, -2, \cdots$.
Further, all these zeros and poles are simple.   
Hence $D(s, \chi, F)$ for $\chi \ne \chi_0$ extends to a 
holomorphic function on $\C$ 
and vanishes at $s =0 , -1 , -2, \cdots$.
 
If $\chi = \chi_0$, then $D(s, \chi_0, F)$ has a meromorphic 
continuation to $\C$ possibly with a simple pole at $s=k$. 
Indeed, the function $D(s, \chi_0, F)$ has a simple real 
pole at $s=k$, since $D^*(s, \chi_0, F)$
has a simple real pole at $s=k$ and none of the functions
$L(2s -2k + 2n, \chi^2)$, $\Gamma(s-k+n)$ and $\Gamma(s)$ have
a zero or a pole at $s=k$ and they are holomorphic there. 
Furthermore, the series $D(s, \chi_0, F)$ 
vanishes at $s = 0, -1, -2, \cdots$. 

Hence using orthogonality of characters, we get

\begin{eqnarray}\label{one}
D(s; a,q, F) &=&
\frac{1}{\varphi(q)} \sum_{m \ge 1} \sum_{\chi \mod q} 
\chi(a^{-1}m) \lan f_m, f_m \ran m^{-s} \nonumber \\
\label{test} &=& \frac{1}{\varphi(q)} \sum_{\chi \mod q} 
\chi(a^{-1})~ D(s, \chi, F) 
\phantom{m} \text{for } \Re(s) > k.
\end{eqnarray}
This implies that the Dirichlet series $D(s ; a,q, F)$ has a 
meromorphic continuation to $\C$
with a simple real pole at $s=k$ and vanishes at 
$s = 0, -1, -2, \cdots$. 
\end{proof}

\begin{rmk}\label{lem-1}
{\rm It is clear from equation (\ref{one}) that 
the residue of $D(s;a,q,F)$ at $s=k$ 
depends only on $q$ and the residue of $D(s,\chi_0,F)$ at $s=k$,
but not on~$a$. In fact, the residue of $D(s,\chi_0,F)$ can be expressed 
in terms of the Petersson scalar product of $F$ with the 
Trace of ``$\chi_0$-twist of $F$'' (see \cite[Thm. 1]{KKS} 
for further details).}    
\end{rmk}

\begin{rmk}\label{imp}
{\rm 
The residue of $D(s;a,q,F)$ at $s=k$ is real and positive. This 
is true as $D(s;a,q,F)$ is holomorphic on $\Re(s) > k$ with
a simple pole at $s=k$ and is positive 
on the real half axis $\Re(s) > k$.}
\end{rmk}

We end this section by recalling a recent result of Pribitkin 
on Dirichlet series with oscillating coefficients.

\smallskip
\noindent
{\bf Definition.}
We call a sequence $\{ a_n \}_{n=1}^{\infty}$ with $a_n\in\R$ 
oscillatory if there exist infinitely many $n$ such that $a_n >0$ 
and infinitely many $n$ such that $a_n<0$.

\begin{thm}[Pribitkin \cite{WP},\cite{WP1}] \label{Pribitkin-result}
Let $a_n$ be a sequence of real numbers such that 
the associated Dirichlet series
$$
F(s) =\sum_{n=1}^{\infty} \frac{a_n}{n^s}
$$ 
be non-trivial and it converges on some half-plane.  
If $F(s)$ is holomorphic on the whole real line and has 
infinitely many real zeros, then the sequence 
$\{a_n\}_{n=1}^{\infty}$ is oscillatory.
\end{thm}

\smallskip

\section{Statement and proof of the Main Result}

\smallskip

Let $a, q > 1$ be natural numbers with $(a,q)=1$. Also let 
$c_{F,q}$ be the residue of the function $D(s;a,q,F)$ at $s=k$.
Recall that $c_{F,q}$ is independent of $a$ by Remark~\ref{lem-1}.
Moreover, $c_{F,q}$ is real and positive by Remark \ref{imp}.

\begin{thm}\label{main-thm}
Let $a, q > 1$ be natural numbers with $(a,q)=1$ and
$F$ be a non-zero Siegel cusp form in $S_k(\Gamma_n)$, $n>1$ 
with Fourier-Jacobi coefficients
$\{f_m\}_{m \in \N}$. Then there exist 
infinitely many  $m$ with $m \equiv a \mod q$ 
such that $\lan f_m, f_m\ran >  c_{F,q} m^{k-1}$. 
\end{thm}

\begin{proof}
Consider the Dirichlet series
\begin{equation}
\label{key-equation}
\overline{D}(s; a,q, F) =  D(s ; a, q,F) -  
c_{F,q} \zeta(s-k+1) \phantom{m} \text{for  } \Re(s) > k.
\end{equation}

By Proposition~\ref{Main-Prop}, the series 
$\overline{D}(s; a,q, F)$ has a 
meromorphic continuation to $\C$ with
no poles on the real line and vanishes at $s= k-1 - 2t$, where 
$t \in \N, ~ t > (k-1)/2 $. 

For $m \ge 1$, let 
\begin{eqnarray}
\label{key-expression}
\beta(m) &:=& \left\{ \begin{array}{ll}
                  \lan f_m, f_m\ran - c_{F,q} m^{k-1}
                    & \mbox{if $m \equiv a \!\!\!\pmod{q}$} \\
                   - c_{F,q} m^{k-1}   & \mbox{otherwise}
                  \end{array} \right.  
\end{eqnarray}
be the general coefficient of $\overline{D}(s; a, q,F)$. 
We know that $\overline{D}(s;a,q,F)$ cannot be 
identically zero as $c_{F,q} > 0$ by Remark \ref{imp}.
Then by using Theorem~\ref{Pribitkin-result}, there exist 
infinitely many $m$
with $m \equiv a \pmod{q}$ such that 
$\lan f_m, f_m\ran > c_{F,q} m^{k-1}$.
\end{proof}

Using the above method, we are unable to prove that
there exist infinitely many $m$
with $m \equiv a \pmod{q}$ such that 
$\lan f_m, f_m\ran  < c_{F,q} m^{k-1}$.
But we can prove the following weaker
theorem.

\begin{thm}
Let $F$ be a non-zero cusp form in $S_k(\Gamma_n), n>1$ with 
Fourier-Jacobi coefficients $\{f_m\}_{m \in \N}$. 
Let $q$ be a natural number and also let $c_{F,q}$ be the 
residue of $D(s;a,q,F)$ for some $a \in \N$ with $(a,q)=1$
{\rm(}hence for all $a${\rm)}. Then 
there exist natural numbers $b,c$ with $(bc,q)=1$ 
such that the following hold:
\begin{itemize}
\item 
there exist infinitely many $m \in \N$ with $m \equiv b \pmod q$
such that $\lan f_m, f_m \ran >  q c_{F,q}m^{k-1}$ and 
\item 
there exist infinitely many $m \in \N$ with $m \equiv c \pmod q$
such that $\lan f_m, f_m \ran < qc_{F,q} m^{k-1}$.
\end{itemize}
\end{thm}

\begin{proof}
Consider the Dirichlet series
\begin{eqnarray*}
\overline{D}(s; a,q, F) :=
\sum_{a=1 \atop (a,q)=1}^{q-1}D(s; a, q, F)  
~-~ \alpha_{F,q}M \zeta(s-k+1),
\end{eqnarray*}
where 
$$
M:= \prod_{p|q \atop p \text{  prime}}(1 - p^{-(s-k+1)})
\phantom{m}  
\text{ and }
\phantom{m}
\alpha_{F, q}:= qc_{F,q}.
$$
Since 
$$
M \zeta(s-k+1) 
= q^{-(s-k+1)} \sum_{a=1 \atop (a,q)=1}^{q-1}\zeta(s-k+1, a/q),
$$
we have
\begin{eqnarray*}
\overline{D}(s; a,q, F)
&=&
\sum_{a=1 \atop (a,q)=1}^{q-1} \left[ D(s; a, q, F)  
~-~ \alpha_{F,q}q^{-(s-k+1)} \zeta(s-k+1, a/q) \right]\\
&=&
\sum_{a=1 \atop (a,q)=1}^{q-1} ~~\underset{\underset{m \equiv a \mod q}
{m \ge 1}}{\sum} \frac{\lan f_m, f_m \ran - \alpha_{F,q}m^{k-1}}{m^s}. 
\end{eqnarray*}
By Proposition~\ref{Main-Prop}, the series $\overline{D}(s; a, q, F)$ has 
a meromorphic continuation to $\C$ with
no poles on the real line and vanishes at $s= k-1 - 2t$, where 
$t \in \N, ~ t > (k-1)/2 $. 

Note that the function $\overline{D}(s; a, q, F)$ can
not be identically zero. If otherwise, 
$$
\sum_{a=1 \atop (a,q)=1}^{q-1} D(s; a, q, F) 
= \alpha_{F,q} M \zeta(s-k+1).
$$
This is a contradiction as zeros of the 
Riemann zeta function on the negative real axis are at 
negative even integers whereas each $D(s;a,q, F)$ 
has zeros at all negative integers. Now using 
Theorem~\ref{Pribitkin-result}, 
we get the desired result.
\end{proof}

As an immediate corollary, we get
\begin{cor}
For $n>1$, let $F$ be a non-zero cusp form in 
$S_k(\Gamma_n)$ with Fourier-Jacobi coefficients
$\{f_m\}_{m \in \N}$. Let $c_{F,2}$ be the 
residue of the series $D(s; 1,2, F)$ at $s=k$. 
Then the following hold:
\begin{itemize}
\item 
there exist infinitely many odd $m \in \N$ 
such that $\lan f_m, f_m \ran > 2 c_{F,2} m^{k-1}$ and 
\item 
there exist infinitely many odd $m \in \N$ 
such that $\lan f_m, f_m \ran < 2 c_{F,2} m^{k-1}$.
\end{itemize}
\end{cor}

\smallskip
\noindent
{\bf Acknowledgments.}
The second author would like to thank the Hausdorff Research 
Institute for Mathematics, where most of this work was 
carried out during the Trimester 
program ``Arithmetic and Geometry'', Jan-Feb 2013.
We would like to thank the referee for several relevent 
suggestions which improved the presentation
of the paper. Further, we would like to thank W. Kohnen 
for making this collaboration possible.

\end{document}